\documentclass[12pt]{amsart}
\usepackage{amssymb,amsmath,amsfonts,latexsym,mathrsfs}
\usepackage{bm}

\usepackage{array,graphics,color}
\numberwithin{equation}{section}

\newtheorem{propn}{Proposition}[section]
\newtheorem{thm}[propn]{Theorem}
\newtheorem{lemma}[propn]{Lemma}
\newtheorem{cor}[propn]{Corollary}
\newtheorem*{thm*}{Theorem}
\theoremstyle{definition}
\newtheorem{defn}[propn]{Definition}

\newtheorem{rem}{Remark}[section]

\newcommand{\Nat}{\mathbb{N}}

 \newcommand{\D}{\mathbb{D}}

\newcommand{\clb}{\mathcal{B}}

\newcommand{\cld}{\mathcal{D}}
\newcommand{\cle}{\mathcal{E}}
\newcommand{\clf}{\mathcal{F}}

\newcommand{\clh}{\mathcal{H}}
\newcommand{\clk}{\mathcal{K}}

\newcommand{\clm}{\mathcal{M}}
\newcommand{\cln}{\mathcal{N}}

\newcommand{\clq}{\mathcal{Q}}

\newcommand{\clw}{\mathcal{W}}

\newcommand{\C}{\mathbb{C}}

\newcommand{\raro}{\rightarrow}

\newcommand{\Hil}{\mathcal{H}}

\begin{document}

\title[Pairs of Commuting Pure Contractions and Isometric dilation]{Pairs of Commuting Pure Contractions and Isometric Dilation}


\author[Sarkar]{Srijan Sarkar}
\address{Department of Mathematics, Indian Institute of Science, Bangalore, 560012, India}
\email{srijans@iisc.ac.in,
srijansarkar@gmail.com}



\begin{abstract}
A particular case of the fundamental Sz.-Nagy--Foias functional model for a contraction states that a pure contraction always dilates to a pure isometry. We are interested in the similar question for pairs, more precisely: \textit{does a pair of commuting pure contractions always dilate to a pair of commuting pure isometries?} The purpose of this article is to identify pairs of commuting pure contractions for which the above question has an affirmative answer. Our method is based on an explicit structure of isometric dilation for pure pairs of commuting contractions obtained in a recent work by Das, Sarkar and the author.

\end{abstract}
\maketitle

\section{Introduction}
A series of important developments in operator theory evolved from the foundational result of Sz.-Nagy on the unitary dilation of contractions (\cite{Nagy}); the most remarkable being the Sz.-Nagy and Foias functional model of a contraction $T$  (that is, $\|T\| \leq 1$) on a Hilbert space $\clh$ (see \cite{NF}). In this article, we are interested in the model of a particular class of contractions. To be more precise, let us recall the following definition.
\begin{defn}
A contraction $T$ is said to be pure if $T^{*m} \raro 0$ in the strong operator topology (that is, for all $h \in \clh, \|T^{*m}h\| \raro 0)$ as $m \raro \infty$.
\end{defn}
The Sz.-Nagy-Foias model for pure contractions (see \cite{NF}) states that
\begin{thm}\label{NF}(\textsf{Sz.-Nagy and Foias})
Let $T$ be a pure contraction on $\clh$. Then
$T$ and $P_{\clq} M_z|_{\clq}$ are unitarily equivalent, where
$\clq$ is a closed $M_z^*$-invariant subspace of a vector-valued
Hardy space $H^2({\cld})$. More precisely, there exists an isometry $\Pi:\clh \raro H^2(\cld)$ such that
\[
\Pi T^* = M_z^* \Pi.
\]
\end{thm}
Here the $\cld$-valued Hardy space (for a Hilbert space $\cld$) on the unit disc $\D$ is defined by
\[
H^2({\cld}):= \{ f = \sum_{{k} \in \mathbb{N}} \eta_{{k}} z^{{k}}
\in \mathcal{O}(\D, \cld): \eta_j \in \cld, j \in \mathbb{N},
\|f\|^2 : = \sum_{{k} \in \mathbb{N}} \|\eta_{{k}}\|^2 < \infty\}.
\]
The scalar valued Hardy space of functions on $\D$ is denoted by $H^2$ and the collection of bounded analytic functions over unit disc is denoted by $H^{\infty}$. 

An equivalent way of looking at Theorem \ref{NF} is as follows: every pure contraction dilates to a pure isometry. This article is devoted to the study of this characterization for pairs of commuting pure contractions $(T_1,T_2)$ on Hilbert spaces $\clh$.

We would like to note that our study differs from the approach that directly applies the paradigm of Theorem \ref{NF} to the multivariate setting. The latter paradigm would require commuting contractions to admit isometric dilation to shift operators on vector-valued Hardy spaces on the unit polydisc $\D^n$. Although shift operators are the simplest examples of commuting pure isometries, they satisfy certain natural conditions, which are reflected as conditions on the given tuple of commuting contractions to ensure the existence of a dilation. For instance, Curto and Vasilescu in \cite{CV}, remarkably, observed that a $n$-tuple of commuting pure contractions $T=(T_1,\ldots,T_n)$ on $\clh$ dilates to the shift operators $(M_{z_1},\ldots,M_{z_n})$ on some vector-valued Hardy space on $\D^n$ if it satisfies the positivity condition
\begin{equation}\label{positive}
\mathbb{S}_n^{-1}(T,T^*) \geq 0.
\end{equation}
Here, $\mathbb{S}_n(\bm{z},\bm{w}):= \prod_{i=1}^n (1 -z_i \bar{w_i})^{-1}$, for all $\bm{z},\bm{w} \in \D^n$, is the Szeg\"o kernel for the Hardy space on $\D^n$. In a recent work \cite{BDHS}, Barik et al. proved, by using results in \cite{CV} and certain techniques from \cite{DS} and \cite{Fac}, that a wider class of tuples of commuting contractions admit dilation (in the polydisc setting) to a tractable and explicit tuple of commuting isometries. The class of commuting contractions studied in \cite{BDHS} are based on the tuples considered by Grinshpan et al. in \cite{DKVW}, where they first observed that such tuples satisfy von Neumann inequality on $\D^n$ and, consequently, admit unitary dilation. We refer the reader to the references in \cite{BDHS} and \cite{DKVW}, for a detailed list of results in this direction of research.

Based on the above discussion, (potentially) one sees that there are now approaches that reach beyond the condition (\ref{positive}). Thus, our interest in the dilation problem for pure contractions stems out from the following viewpoints:
\begin{enumerate}
\item[(i)] To identify pairs of commuting pure contractions, which do not necessarily satisfy positivity conditions like (\ref{positive}), but dilate to pairs of commuting pure isometries.
\item[(ii)] Establish the isometric dilation of such pairs on vector-valued Hardy spaces on the unit disc instead of the bidisc.
\end{enumerate}

The latter goal is related to the study of functional models for pairs of commuting contractions, which is far from being complete. The main motivation behind our interest comes from a recent work by Das, Sarkar and the author in \cite{Fac}, where an explicit structure of isometric dilations for pure pairs of commuting contractions had been established. The isometries appearing in this result belong to the class of pure pairs of commuting isometries and their characterization is due to Berger, Coburn and Lebow (abbreviated as BCL). These isometries are the object of our interest and play a central role in our study. Before going into further details, let us recall the definition of pure tuples of operators.

\begin{defn}
A tuple of commuting contractions (isometries) $(T_1,\ldots,T_n)$ on $\clh$ is said to be pure if $T:=\prod_{j=1}^n T_j$ is a pure contraction (isometry).
\end{defn}

In the case of a single isometry, there is a complete description due to Wold (\cite{Wold}) and von Neumann (\cite{Neumann}) which states that: an isometry $V$ on $\clh$ is unitarily equivalent to the direct sum of a unitary $U$ on the space $\clh_{U} = \bigcap_{k \in \Nat} V^k \clh$ and $M_z$ on $H^2(\clw(V))$, where $\clw(V):= (I_{\clh} - VV^*) \clh$. If $V$ is a pure isometry then $\clh_U = \{0\} $ and thus, $V$ becomes unitarily equivalent to $M_z$ on $H^2(\clw(V))$. In the multivariate setting, the structure of commuting tuples of isometries is complicated and till now there has not been a complete description (see \cite{BDF1, BDF2, BDF3, Gaspar, MSS, Y1} for some important progress in this direction). However, this problem has been addressed to a great extent by BCL's characterization (see \cite{BCL}) . Their result for a pure pair of commuting isometries states that:
\begin{thm}\label{BCLthm}(\textsf{Berger, Coburn and Lebow})
A pair of commuting isometries $(V_1,V_2)$ on a Hilbert space $\clh$ is pure if and only if there exists a Hilbert space $\cle$, a unitary $U$ in $\clb(\cle)$ and a orthogonal projection $P$ on $\clb(\cle)$ such that $(V_1,V_2)$ on $\clh$ is unitarily equivalent to a commuting pair of isometries $(M_{\Phi},M_{\Psi})$ on $H^2(\cle)$, where
\begin{equation}\label{BCLform}
\Phi(z) =  (P+zP^{\perp})U^* ; \quad \Psi(z) = U(P^{\perp}+zP)\quad(z \in \D).
\end{equation}
\end{thm}
It follows that $M_{\Phi}M_{\Psi} = M_{\Psi}M_{\Phi} = M_z$. Inspired by the above result, an ordered triple of spaces and operators $(\cle,U,P)$ is said to be a BCL triple if $\cle$ is a Hilbert space, $U$ is a unitary and $P$ is an orthogonal projection on $\cle$, respectively. We shall say that $(M_{\Phi},M_{\Psi})$ is a pair of isometries associated to a BCL triple $(\cle,U,P)$ if and only if the symbols $\Phi(z),\Psi(z)$ are in the form of (\ref{BCLform}).

The BCL structure was developed by the authors to study representation and Fredholm theory for C*-algebras generated by commuting isometries. However, in recent times, the BCL structure has found applications in understanding the structure of isometric dilation for commuting contractions (see \cite{BDHS, MB, Fac, Sau}). Following this, there has been an important development on a functional model for pair of commuting contractions by Ball and Sau (\cite{BS}). The BCL structure has also been used by Bhattacharyya, Kumar and Sau to obtain some remarkable new characterizations for distinguished varieties in the bidisc, polydisc and the symmetrized bidisc (see \cite{BKS}).

Our work in this article begins by analysing pure tuples of commuting isometries \\ $(V_1,\ldots,V_n)$. In section \ref{sec-pureisom}, we find a necessary and sufficient condition for an isometry $V \in (V_1,\ldots,V_n)$ to become a pure isometry (see Proposition \ref{compure1}). Applying this result to isometries $(M_{\Phi}, M_{\Psi})$ associated to a BCL triple $(\cle,U,P)$, we obtain that if
\[
U = \begin{bmatrix}
A & B \\C & D
\end{bmatrix} : P^{\bot} \cle \oplus P \cle \raro P^{\bot} \cle \oplus P \cle,
\]
then $M_{\Phi}, M_{\Psi}$ are pure isometries if and only if $A$ and $D^*$ are pure contractions, respectively (see Theorem \ref{compure3}). Thus, diagonal entries of $U$ play an important role throughout our study.

Returning to the work in \cite{Fac}: the authors have proved that a pure pair of commuting contractions always dilates to pair of isometries $(M_{\Phi},M_{\Psi})$ associated to some BCL triple $(\cle,U,P)$. The unitary $U$ appearing in the BCL triple is exactly the unitary that had first appeared in And\^o's remarkable generalization of Sz.Nagy's theorem for contractions (see \cite{An}). The need for the existence of $U$ came from a compatibility condition that was essential for the existence of commuting unitary dilations. In section \ref{sec-purecont}, we show that this unitary has several other important features when we consider a pair of commuting pure contractions $(T_1,T_2)$, to begin with. To establish these results, we focus on an abstract setting, where such pairs $(T_1,T_2)$ have a canonical BCL triple of the form $(\cle,U,P)$ associated to them. In this situation, Theorem \ref{purecrit} shows that the diagonal entries $A$ and $D^*$ become completely non co-isometric contractions (see Section \ref{sec-prelim} for definitions). This result has an immediate effect on the symbols of the isometries $(M_{\Phi},M_{\Psi})$. More precisely, the pair of partial isometries $(\Phi(0),\Psi(0))$ always possesses the \textit{wandering subspace property} (see  Definition \ref{wander} and Corollary \ref{defects}).

Using the properties of the diagonal elements of the unitary $U$ we establish the main results of this article. In section \ref{sec-dilcont}, we prove that in the following cases, $(T_1,T_2)$ always dilates to a pair of commuting pure isometries:
\begin{enumerate}
\item[(i)] $(T_1,T_2)$ such that $\mbox{dim }\overline{(I - T_iT_i^*)^{\frac{1}{2}}\clh} < \infty$ for $i =1,2$. 
\item[(ii)] $(T_1,T_2)$ is a pair of commuting pure partial isometries. 
\end{enumerate}
The proof of statement $(i)$ as can be seen in Theorem \ref{purecont1}, is a direct consequence of the fact that in this situation, there always exists a BCL triple associated to $(T_1,T_2)$, and the diagonal entries of the unitary satisfy the pure conditions in Theorem \ref{compure3}. On the other hand, statement $(ii)$ requires a new technique to extend the result from pairs with BCL triples to arbitrary pairs (see Theorem \ref{purecont2}). We end this section with another characterization which involves an approach  different from the results in section \ref{sec-purecont}. More precisely, if we consider unitaries of the form 
\[
U = \begin{bmatrix}
0 & W_1 \\ W_2 & 0
\end{bmatrix} : P^{\bot} \cle \oplus P \cle \raro P^{\bot} \cle \oplus P \cle,
\]
then from Theorem \ref{compure3} we observe that $(\cle,U,P)$ contributes to pure isometries $(M_{\Phi},M_{\Psi})$. In Theorem \ref{normpure}, we identify pairs $(T_1,T_2)$ for which such a unitary can be constructed and moreover, admits dilation to $(M_{\Phi},M_{\Psi})$. This turn out to be pairs satisfying certain certain norm conditions on the defect operators $D_{T_i^*}: = (I - T_iT_i^*)^{\frac{1}{2}}$, for $i=1,2$.

In section \ref{sec-vnI}, we obtain a von Neumann inequality for pairs $(T_1,T_2)$ appearing in statement $(i)$ of the above discussion. Furthermore, we show that our characterization recovers the von Neumann inequality obtained by Agler, McCarthy in \cite{AM1} and also the extended result by Das, Sarkar in \cite{DS}.

Section \ref{sec-prelim} serves as a prelude to all the other sections. It contains the preliminary results, definitions and notations required throughout this article.

\section{Notation and Preliminaries}\label{sec-prelim}
In this section, we set notations and recall definitions and some results. We also prove some results essential for this article.  The monographs by Agler, McCarthy \cite{AM2} and Bercovici, Foias, K\'erchy, Sz.-Nagy \cite{NF} are excellent references for most of the details in this section.

Let us begin by recalling some properties of vector-valued Hardy spaces. The Hardy space $H^2(\cle)$ is also a reproducing kernel Hilbert space with the reproducing kernel function \\$K:\D\times \D \raro \clb(\cle)$ defined by
\[
K(z,w) = s(z,w) I_{\cle} \quad (z,w \in \D),
\]
where $s(z,w)=(1-z\bar{w})^{-1}$ is the Szeg\"o kernel for the scalar Hardy space $H^2$. More precisely, for all $f \in H^2(\cle)$ the following reproducing property is satisfied
\[
\langle f(w), \eta \rangle = \langle f, K(\cdot,w)\eta \rangle   \quad (w \in \D, \eta \in \cle).
\]
It follows that $\{ K(\cdot,w) \eta: w \in \D, \eta \in \cle \}$ is a total set for $H^2(\cle)$. Let $H^{\infty}(\clb(\cle))$ denote the collection of $\clb(\cle)$-valued bounded analytic functions on $\D$, where $\clb(\cle)$ is the set of  bounded operators on Hilbert space $\cle$. It follows from the closed graph theorem that $\Phi \in H^{\infty}(\clb(\cle))$ induces a bounded multiplication operator $M_{\Phi} \in \clb(H^2(\cle))$, which for all $w \in \D$ and $\eta \in \cle$ satisfies,
\[
M_{\Phi}^* (K(\cdot,w) \eta) = K(\cdot,w) \Phi(w)^* \eta.
\]
The above identity is then used to show that $\|M_{\Phi}\| = \mbox{sup}_{z \in \D} \{\|\Phi(z)\|_{\clb(\cle)} \}$.

An important way to construct multiplication operators on Hardy spaces is through the method of transfer functions (cf. \cite{AM2}). Let $\clh_1$ and $\clh_2$ be two  Hilbert spaces, and 
\[U = \begin{bmatrix}A&B\\C&D\end{bmatrix} \in \clb(\clh_1 \oplus
\clh_2),
\]
be a unitary operator. Then the $\clb(\clh_1)$-valued analytic function $\tau_U$ on $\mathbb{D}$ defined by 
\[\tau_U (z) := A + z B (I-z D)^{-1} C \quad \quad (z \in \D), \]
is called the \textit{transfer function} of $U$. Using $U^* U = I$,
a standard and well known computation yields (cf. \cite{DS})
\begin{equation}\label{isometry}
I- \tau_U (z)^*\tau_U (z) = (1- \mid z\mid^2) C^*(I-\bar{z} D^*)^{-1}(I-z
D)^{-1} C \quad \quad (z\in\D).
\end{equation}
This identity shows that $\tau_{U}$ is a contractive operator valued analytic function or equivalently, $\|M_{\tau_U}\| \leq 1$. In particular, if we assume the Hilbert spaces $\clh_1,\clh_2$ to be finite dimensional, then $\tau_{U}$ becomes a matrix-valued rational inner function and therefore, $M_{\tau_U}$ becomes an isometry (see page 138, \cite{AM2}). Hardy spaces and multiplication operators play an important role in the a functional model for contractions. Before going into further details, let us recall the definition of dilation of contractions.

\begin{defn}
A tuple of operators $V=(V_1,\ldots,V_n)$ on $\clk$ is said to be a dilation of $T=(T_1,\ldots,T_n)$ on $\clh$ if there exists an isometry $\Pi:\clh \raro \clk$ such that $\Pi T_i^* = V_i^* \Pi$, for $i=1,\ldots,n$. A dilation $(V,\clk)$ of $(T,\clh)$ is said to be minimal if $\clk = \overline{\mbox{span}} \{V^{\bm k}h: \bm{k} \in \Nat^n, h \in \clh\}$.
\end{defn}

In the sequel, we will review the minimal isometric dilation of a pure contraction. For a contraction $T$ (equivalently, $I_{\clh} - TT^* \geq 0$) recall that the defect operator is $D_{T} := (I_{\clh} - T^* T)^{\frac{1}{2}}$, and the corresponding defect spaces are
\[
\cld_{T} := \overline{D_T \clh}; \quad  \cld_{T^*} := \overline{D_{T^*} \clh}.
\]

If $T$ is pure, then the isometry $\Pi_T : \clh \raro H^2(\cld_{T^*})$ defined by, 
\begin{equation}\label{dil-def}
(\Pi_T h)(z) = D_{T^*}(I_{\clh} - z T^*)^{-1}h \quad \quad (z \in \D, h
\in \clh),
\end{equation}
satisfies $\Pi_T T^* = M_z^* \Pi_T$ (cf. \cite{JS}). Thus, proving that $T$ dilates to $M_z$ on $H^2(\cld_{T^*})$. This isometric dilation is moreover minimal, that is,
\begin{equation}\label{min}
H^2(\cld_{T^*}) =
\overline{\mbox{span}} \{z^m \Pi_{T} f : m \in \mathbb{N}, f \in \clh\},
\end{equation}
and hence unique in an appropriate sense (see \cite{NF}). In particular, this discussion yields a proof of Theorem \ref{NF} by showing that every pure contraction is unitarily equivalent to $P_{\clq}M_z|_{\clq}$, where $\clq:= \Pi_T \clh \subseteq H^2(\cld_{T^*})$.

For the purpose of this article, we will now focus on orthogonal decompositions for contractions. The most fundamental result on such decompositions is due to Sz.-Nagy, Foias and Langer which states that a contraction $T$ is always the orthogonal direct sum of a unitary and a \textit{completely non unitary} (c.n.u.) contraction. A contraction $T$ on $\clh$ is c.n.u. if and only if it does not have any unitary summand. In particular, the subspace
\[
\clh_U (T^*) := \{h \in \clh: \|T^{*n}h\| = \|h\| = \|T^n h\| \text{ for all }n \in \Nat\},
\]
which is also the largest reducing subspace on which $T$ is a unitary should be equal to $\{0\}$ (see \cite{NF}). By definition, it clearly follows that $\clh_U(T) = \clh_{i}(T) \cap \clh_{i}(T^*)$, where
\[
\clh_{i}(T) := \{h \in \clh: \|T^{n}h\| = \|h\| \text{ for all }n \in \Nat\}.
\]
This is also a closed subspace due to the following fact
\[
\clh_{i}(T) = \bigcap_{n \in \Nat}\mbox{ker} D_{T^{n}} = \{h \in \clh: T^{*n}T^{n}h = h \text{ for all }n \in \Nat\}.
\]
$\clh_{i}(T^*)$ is also the largest $T^*$-invariant subspace on which $T$ is a co-isometry (cf. \cite{Levan}). A contraction $T$ is said to be a \textit{completely non co-isometric} (c.n.c.) contraction if and only if $\clh_{i}(T^*) = \{0\}$. For a c.n.u. contraction there always exists an orthogonal decomposition of the Hilbert space $\clh$, with respect to the c.n.c. part. This the content of  [Lemma 2.2, \cite{ZYL}] which states that for a c.n.u. contraction $T \in \clb(\clh)$ we always have
\begin{equation}\label{orth}
\clh = \bigvee_{k \in \Nat} T^k \cld_{T^*} \oplus \clh_{i}(T^*).
\end{equation}

In particular, for a c.n.c. contraction $T$, the above orthogonal decomposition of $\clh$ implies that $\clh = \bigvee_{k \in \Nat}T^k \cld_{T^*}$. For partial isometries this result has an important implication. Recall that a partial isometry $T$ on $\clh$ is defined to be a bounded operator that is isometric on the orthogonal complement of its kernel. An equivalent criterion for $T$ to be a partial isometry is
\begin{equation}\label{char_pi}
\mbox{ker }T^* = \cld_{T^*}.
\end{equation}

Before stating the implication, let us recall the definition of \textit{wandering subspace property} as given by Shimorin (see \cite{Shimorin}).
\begin{defn}\label{wander}
A bounded operator $T \in \clb(\clh)$ is said to possess the wandering subspace property if $\clh = \bigvee_{k \in \Nat} T^k \mbox{ ker } T^*$.
\end{defn}

The following characterization is a direct consequence of condition (\ref{orth}) and (\ref{char_pi}).
\begin{propn}\label{partial_wander}
A c.n.u. partial isometry possesses the wandering subspace property if and only if it is a c.n.c. contraction.
\end{propn}

We will now prove a result concerning transfer functions and c.n.c. contractions, which may be of independent interest. 
\begin{propn}\label{transfer_cnu}
Let $\tau_U(z)$ be a transfer function corresponding to a unitary $U \in \clb(\clh_1 \oplus \clh_2)$. If $\tau_U(0)$ is a c.n.c. contraction then $\tau_U(\lambda)$ is a c.n.u. contraction for all $\lambda \in \D$.
\end{propn}
\begin{proof}
For a $\lambda \in \D$, let us assume that $\clm$ is a $\tau_U(\lambda)$-reducing subspace of $\clh_1$ such that $\tau_U(\lambda)|_{\clm}$ is a unitary. More precisely,
\[
P_{\clm} (I_{\clh_1 \oplus \clh_2} - \tau_U(\lambda)^* \tau_U(\lambda))P_{\clm}=0
\]
and 
\[
P_{\clm} (I_{\clh_1 \oplus \clh_2} - \tau_U(\lambda) \tau_U(\lambda)^*)P_{\clm}=0
\]
From condition (\ref{isometry}) we have
\begin{align*}
P_{\clm} (1- |\lambda|^2) C^*(I-\bar{\lambda} D^*)^{-1}(I-\lambda D)^{-1} C P_{\clm} = P_{\clm} (I_{\clh_1 \oplus \clh_2} - \tau_U(\lambda)^* \tau_U(\lambda))P_{\clm} = 0.
\end{align*}
which implies that $C P_{\clm} = 0$. Using this result in the second identity we get 
\begin{align*}
P_{\clm} (I_{\clh_1 \oplus \clh_2} - \tau_U(\lambda) \tau_U(\lambda)^*) P_{\clm} &= P_{\clm}  - \tau_U(\lambda) P_{\clm} \tau_U(\lambda))^*\\
&= P_{\clm}  - (A + \lambda B (I-\lambda D)^{-1} C)P_{\clm}(A^* + \bar{\lambda} C^* (I- \bar{\lambda} D^*)^{-1} B^*)\\
&=  P_{\clm}  - A P_{\clm}A^* .
\end{align*}
Thus, we have $A P_{\clm}A^* = P_{\clm}$. Since, $U$ is a unitary we have $A^*A+C^*C = I_{\clh}$ and therefore,
\[
P_{\clm}A = A P_{\clm}A^*A = A P_{\clm}(I_{\clh_1} - C^* C) = A P_{\clm}.
\]
Hence, $\clm$ is a $A$-reducing subspace of $\clh_1$ and furthermore, $ A P_{\clm} (AP_{\clm})^* = P_{\clm}$. This implies that $A$ is a co-isometry on the reducing subspace $\clm$, which is a contradiction to the assumption that $A$ is a c.n.c. contraction and thus $\clm=\{0\}$. This completes the proof.
\end{proof}

Apart from the above properties of pure contractions, partial isometries and transfer functions, our study involves a specific contraction associated to tuples of isometries. More precisely, let $(V_1,\ldots,V_n)$ be a pure tuple of commuting isometries on $\clh$ and let $\hat{V_j} := \prod_{k=1, k \neq j}^n V_k$.
\begin{defn}\label{fringe}
For any $j \in \{1,\ldots,n\}$, the fringe operator $F_{V_j}: \clh \ominus \hat{V_j} \clh \raro \clh \ominus \hat{V_j} \clh$ with respect to the isometry $V_j$ is defined by $F_{V_j} := P_{\clh \ominus \hat{V_j} \clh} V_j |_{\clh \ominus \hat{V_j} \clh}$.
\end{defn} 
Here, $P_{\clh \ominus \hat{V_j} \clh}$ denotes the orthogonal projection onto the closed subspace $\clh \ominus \hat{V_j} \clh$. The fringe operator for pairs of commuting isometries has been studied extensively by Bercovici, Douglas and Foias (in \cite{BDF3}) under the name of a ``pivotal'' operator. A remarkable result in their paper states that: every contraction on a separable Hilbert space, up to unitary equivalence, is the fringe operator corresponding to a pure pair of commuting isometries. We refer the reader to papers \cite{Yang, MSS} for more results on fringe operators and pairs of commuting isometries. 

\section{Commuting pure isometries}\label{sec-pureisom}
In this section, our aim is to find conditions on BCL triples $(\cle,U,P)$ such that the corresponding BCL type isometries $M_{\Phi}$, $M_{\Psi}$ become individually pure.

Let $(V_1,\ldots,V_n)$ be a $n$-tuple of commuting isometries on $\clh$. Note that $V_j$ is an isometry implies that the operator inequality $V_j^m (I_{\clh} - V_j V_j^*)V_j^{*m} \geq 0$ holds for each $m \in \Nat$. Hence, $\{V_j^m V_j^{*m}\}_{m=0}^{\infty}$ is a monotonically decreasing sequence of positive-semidefinite contractions. Thus, there exists a positive contraction (denoted by) $A_{V_j}$ such that $\mbox{SOT--}\lim_{m \raro \infty} V_j^m V_j^{*m} = A_{V_j}^2$. It follows from this convergence that $V_j A_{V_j}^2 V_j^* =  A_{V_j}^2$ and since $V_j$ is an isometry, it follows that $V_j A_{V_j}^2 =   A_{V_j}^2 V_j$. Moreover, the following identities (see Proposition 1, \cite{Duggal})
\[
 A_{V_j}^4 = \mbox{SOT--} \lim_{m \raro \infty}  A_{V_j}^2 V_j^{m}V_j^{*m}  =  \mbox{SOT--} \lim_{m \raro \infty}  V_j^{m}  A_{V_j}^2 V_j^{*m} = \mbox{SOT--} \lim_{m \raro \infty}A_{V_j}^2 = A_{V_j}^2,
\]
prove that $A_{V_j}$ is a projection and therefore,  
\begin{equation}\label{eq_1}
A_{V_j} V_j^* =  V_j^* A_{V_j},
\end{equation}
Also, for all $h \in \clh$ we have
\[
\|A_{V_j}h\|^2 = \langle A_{V_j}h,A_{V_j}h\rangle = \langle A_{V_j}^2h,h\rangle = \lim_{m \raro \infty} \langle V_j^m V_j^{*m}h,h\rangle = \lim_{m \raro \infty} \|V_j^{*m}h\|^2,
\] 
which also implies that $\|A_{V_j}h\| = \|A_{V_j}V_j^{*k}h\|$ for all $k \in \Nat$. We are now ready to observe one of the main results of this section. 

\begin{propn}\label{compure1}
Let $(V_1,\ldots,V_n)$ be a pure tuple of commuting isometries on $\clh$ (that is, $V = \Pi_{k=1}^n V_k$ is a pure isometry on $\clh$). Then for any $j \in \{1,\ldots,n\}$, $V_{j}$ is a pure isometry if and only if the fringe operator $F_{V_j}$ is a pure contraction on $\clh \ominus \hat{V_j} \clh$.
\end{propn}
\begin{proof}
Let us begin by proving the necessary part. Since $(V_1,\ldots,V_n)$ is a commuting tuple therefore, $\hat{V_j} \clh$ is a $V_j$-invariant closed subspace of $\clh$ which further implies that $\clh \ominus \hat{V_j} \clh$ is a $V_j^*$-invariant subspace of $\clh$. Hence,
\[
F_{V_j}^* = \big(P_{\clh \ominus \hat{V_j} \clh} V_j P_{\clh \ominus \hat{V_j} \clh}\big)^* = P_{\clh \ominus \hat{V_j} \clh} V_j^* P_{\clh \ominus \hat{V_j} \clh} = V_j^* P_{\clh \ominus \hat{V_j} \clh}.
\] 
Now, if $V_j$ is a pure isometry, then for all $h \in \clh \ominus \hat{V_j} \clh$ we have
\[
\lim_{m \raro \infty} \|F_{V_j}^{*m} h\| = \lim_{m \raro \infty} \|{V_j}^{*m} P_{\clh \ominus \hat{V_j} \clh}h\| = 0,
\] 
which implies that $F_{V_j}$ is a pure contraction on $\clh \ominus \hat{V_j} \clh$. For proving the other direction, we need to show that $A_{V_j}=0$. Now, by assumption $F_{V_j}$ is a pure contraction and therefore, for all $h \in \clh$
\[
\|A_{V_j} P_{\clh \ominus \hat{V_j} \clh} h\| = \lim_{m \raro \infty}  \|V_j^{*m} P_{\clh \ominus \hat{V_j} \clh} h\|  = \lim_{m \raro \infty}  \|F_{V_j}^{*m} P_{\clh \ominus \hat{V_j} \clh}h\| = 0,
\]
which implies that $\clh \ominus \hat{V_j} \clh \subseteq \mbox{ker }A_{V_j}$ and hence, $A_{V_j} \clh \subseteq \hat{V_j} \clh$. Our aim is to show that $A_{V_j} \clh \subseteq \bigcap_{n \in \Nat} \hat{V_j}^n \clh$. We shall prove by the method of induction. Let us assume that $A_{V_j} \clh \subseteq \hat{V_j}^k\clh$ for some $k \in \Nat$. Since $V^k$, $V_j^k$ and $\hat{V_j}^k$ are all isometries  thus,
\[
P_{\clh \ominus V^k \clh} = I_{\clh} - V^k V^{*k} = (I_{\clh} - V_j^kV_j^{*k}) + V_j^k(I_{\clh} - \hat{V_j}^k \hat{V_j}^{*k})V_j^{*k} = P_{\clh \ominus {V_j}^k \clh} + V_j^k P_{\clh \ominus \hat{V_j}^k \clh}V_j^{*k},
\]
and therefore,
\begin{align*}
P_{\clh \ominus V^k \clh} A_{V_j} = P_{\clh \ominus V^k \clh} V_j^k A_{V_j} V_j^{*k}  &= (P_{\clh \ominus {V_j}^k \clh} + V_j^k P_{\clh \ominus \hat{V_j}^k \clh} V_j^{*k}) V_j^k A_{V_j} V_j^{*k}\\
&= V_j^k P_{\clh \ominus \hat{V_j}^k \clh} A_{V_j} V_j^{*k}
=0.
\end{align*}
In particular, we obtain that $A_{V_j} \clh \subseteq V^k \clh$. Now, for all non-zero $h \in \clh$ and $\eta \in \clh \ominus \hat{V_j}\clh$ we have
\begin{align*}
\langle \hat{V_{j}}^{*k} A_{V_j}h, \eta\rangle = \langle \hat{V_{j}}^{*k} V^k V^{*k} A_{V_j}h, \eta\rangle = \langle V_{j}^k \hat{V_{j}}^{*k} A_{V_j} V_{j}^{*k}h, \eta\rangle = \langle \hat{V_{j}}^{*k} A_{V_j} V_{j}^{*k}h, V_{j}^{*k} \eta\rangle .
\end{align*}
Again,
\begin{align*}
\langle \hat{V_{j}}^{*k} A_{V_j} V_{j}^{*k}h, V_{j}^{*k} \eta\rangle = \langle \hat{V_{j}}^{*k} V^k V^{*k} A_{V_j} V_{j}^{*k}h, V_{j}^{*k} \eta\rangle &= \langle V_{j}^k \hat{V_j}^{*k} A_{V_j} V_{j}^{*2k}h, V_{j}^{*k} \eta\rangle \\
&= \langle \hat{V_j}^{*k}A_{V_j} V_{j}^{*2k}h, V_{j}^{*2k} \eta\rangle.
\end{align*}
Iterating this argument for any $m \in \Nat$ many times we get
\begin{align*}
\langle \hat{V_{j}}^{*k} A_{V_j}h, \eta\rangle = \langle \hat{V_j}^{*k}A_{V_j} V_{j}^{*mk}h, V_{j}^{*mk} \eta\rangle.
\end{align*}
Hence, 
\[
| \langle \hat{V_{j}}^{*k} A_{V_j}h, \eta\rangle | \leq \lim_{m \raro \infty} \|h\| \|V_{j}^{*mk} \eta\| = \lim_{m \raro \infty} \|h\| \|F_{V_{j}}^{*mk} \eta\| = 0,
\]
which implies that $\hat{V_j}^{*k} A_{V_j} \clh \subseteq \hat{V_j}\clh$ and therefore, $ A_{V_j} \clh \subseteq \hat{V_j}^{k+1}\clh$. 
So $A_{V_j} \clh \subseteq \bigcap_{k \in \Nat} \hat{V_j}^k \clh = \clh_{U}(\hat{V_j})$ (by Wold-von Neumann decomposition of $\hat{V_j}$). Thus, for all $h \in \clh$ we have
\[
\|A_{V_j}h\| = \|A_{V_j} V_{j}^{*n}h\| = \|\hat{V_j}^{n}\hat{V_j}^{*n}A_{V_j}V_{j}^{*n}h\| = \|\hat{V_j}^{*n} V_j^{*n} A_{V_j} h\| = \|V^{*n} A_{V_j} h\| \quad (n \in \Nat).
\]
Taking limits on both sides as $n \raro \infty$ we get $A_{V_j}h = 0$. This completes the proof.
\end{proof}
In fact, a similar result holds true for contractions. However, the proof is different from the direct approach taken in the previous characterization and uses certain important results from the literature.
\begin{propn}\label{compure2}
Let $(V,T)$ be a pair of commuting contractions on $\clh$ such that $V$ is a pure isometry. Then $T$ is a pure contraction if and only if $P_{\mbox{ker }V^*} T|_{\mbox{ker } V^*}$ is a pure contraction.
\end{propn}
\begin{proof}
First, let us recall that for an isometry $V$, the wandering subspace $\clw(V) $ is $(I - VV^*) \clh$. From the Wold-von Neumann decomposition of a pure isometry we know that there exists a unitary $U: \clh \raro H^2(\clw(V))$ such that $U V = M_z U$. Since $T$ commutes with $V$ therefore, 
\[
U T U^* M_z = U T V U^* =  U V T U^* = M_z U T U^*.
\]
Hence, $U T U^* = M_{\Theta}$ for some $\Theta(z) \in H^{\infty}(\clb(\clw(V)))$ and moreover, $\|M_{\Theta}\| \leq 1$ (see \cite{JS}). As in Proposition \ref{compure1}, we consider the contractions $A_{T}$ and $A_{M_{\Theta}}$ corresponding to the contractions $T$ and $M_{\Theta}$, respectively. Now, for all $h \in \mbox{ker }A_T$
\[
\|A_{M_{\Theta}} U h\| = \lim_{m \raro \infty} \|M_{\Theta}^{*m} U h\| = \lim_{m \raro \infty} \|U T^{*m}h\|  = \|A_T h\| = 0,
\]
implies that $U \mbox{ ker }A_T \subseteq \mbox{ ker }A_{M_{\Theta}}$. Again, if we take $\eta \in \mbox{ ker }A_{M_{\Theta}}$ and replace $h$ by $U^* \eta$ in the above identity, then it implies that $U^* \mbox{ ker }A_{M_{\Theta}} \subseteq \mbox{ ker }A_T $ and therefore, $U \mbox{ ker }A_T = \mbox{ ker }A_{M_{\Theta}}$. Now, $\|A_{M_{\Theta}} M_z^*h\| \leq \|A_{M_{\Theta}}h\|$ for all $h \in \clh$ implies that $\mbox{ ker }A_{M_{\Theta}}$ is a $M_z^*$-invariant closed subspace of $H^2(\clw(V))$ and therefore, by the Beurling-Lax-Halmos theorem (\cite{Halmos}), there exists a Hilbert space $\clf$ and an inner function $\Gamma(z)$ $\in H^{\infty}((\clb(\clf,\clw(V))))$ such that
\begin{equation}\label{beurling}
\mbox{ ker }A_{M_\Theta}= H^2(\clw(V)) \ominus \Gamma H^2(\clf).
\end{equation}
Now we will follow a method, that is motivated from the proof of [Lemma 3.1, \cite{CIL}]. Fix a $\eta \in \clf$ and consider the element $M_{\Gamma} (1 \otimes \eta)$. It clearly follows that $M_z^* M_{\Gamma} (1 \otimes \eta) \in \mbox{ ker }A_{M_{\Theta}} $ and furthermore, $U \in \clb(\clh)$ is a unitary implies that for each $k \in \Nat$, there exists $h_{1,k} \in \mbox{ ker }V^*$ and $h_{2,k} \in V \clh$ such that
\[
M_{\Theta}^{*k} M_{\Gamma} (1 \otimes \eta) = U h_{1,k} \oplus U h_{2,k}.
\]
By acting $M_z^*$ on the left of the previous identity we get
\begin{align*}
M_z^* M_{\Theta}^{*k} M_{\Gamma} (1 \otimes \eta) = M_z^*U h_{1,k} + M_z^* U h_{2,k} &= U V^* h_{1,k} + M_z^* U h_{2,k} \\
&= M_z^* U h_{2,k}.
\end{align*}
Now, $h_{2,k} \in V \clh$ and $M_z^* U h_{2,k} = U V^* h_{2,k}$ implies that $\|M_z^* U h_{2,k}\| = \|h_{2,k}\|$ for all $k \in \Nat$ and therefore, $\lim_{k \raro \infty} \|h_{2,k}\| = \lim_{k \raro \infty} \|M_z^* U h_{2,k}\| =  \lim_{k \raro \infty} \|M_{\Theta}^{*k} M_z^*  M_{\Gamma} (1 \otimes \eta)\| = 0$. Again, for $k,m \in 
\Nat$
\[
\| M_{\Theta}^{*k+m} M_{\Gamma} (1 \otimes \eta)\| = \| M_{\Theta}^{*m}( U h_{1,k} + U h_{2,k})\| \leq \| M_{\Theta}^{*m} U h_{1,k}\| + \| U h_{2,k}\|.
\]
Now, $P_{\mbox{ker }V^*} T|_{\mbox{ker } V^*}$ is a pure contraction implies that $\mbox{ ker }V^* \subseteq \mbox{ker }A_T$ and therefore, $U h_{1,k} \in \mbox{ ker }A_{M_{\Theta}}$. So letting $m \raro \infty $ in the above identity we get that $\| A_{M_{\Theta}} M_{\Gamma} (1 \otimes \eta)\| \leq \| U h_{2,k}\| $. Following this, if we let $k \raro \infty$ then we obtain that $A_{M_{\Theta}} M_{\Gamma} (1 \otimes \eta) = 0$. This implies that $M_{\Gamma} (1 \otimes \eta) \in \mbox{ ker }A_{M_{\Theta}}$ and therefore, $M_{\Gamma} (1 \otimes \eta) =0$. Now $M_{\Gamma}$ is an isometry implies that $\eta=0$. Since $\eta$ was an arbitrary element in $\clf$ thus, we have $\clf = 0$. From condition (\ref{beurling}) we now have $\mbox{ ker }A_{T} = U^* \mbox{ ker }A_{M_{\Theta}}= U^* H^2(\clw(V))  = \clh$ that is, $A_T = 0$. The necessary part follows in a manner similar to the proof of Proposition \ref{compure1}. This completes the proof.
\end{proof}

We would like to highlight some application of the above results. An equivalent version of statement $(i)$ in the following result was first observed by Bercovici, Douglas and Foias in [Lemma 3.4, \cite{BDF3}].
\begin{cor}\label{pure_cor}
Let $\cle$ be a Hilbert space and $\Gamma(z)$ be a $\clb(\cle)$-valued contractive analytic function on $\D$. Then the following holds true:
\begin{enumerate}
\item[(i)] $M_{\Gamma}$ is a pure contraction on $H^2(\cle)$ if and only if $\Gamma(0)$ is a pure contraction on $\cle$.
\item[(ii)] If $\cle= \mathbb{C}$ and $\Gamma(z)$ be a non-constant inner function on $\D$, then $M_{\Gamma}$ is always a pure isometry.
\end{enumerate}
\end{cor}
\begin{proof}
The proof of both the statements depends on the fact that $(M_z,M_{\Gamma})$ is commuting pair and $P_{\mbox{ker }M_z^*} M_{\Gamma}|_{\mbox{ker }M_z^*} = P_{\cle} M_{\Gamma}|_{\cle}$ is a pure contraction if and only if $\Gamma(0)$ is a pure contraction on $\cle$. In the setting of statement $(ii)$ this  holds true because the scalar-valued inner function always satisfies $|\Gamma(z)|<1$ for all $z \in \D$.
\end{proof}
For the purpose of this article, we shall now focus only on pairs of commuting isometries. We know from the BCL characterization (Theorem \ref{BCLthm}) that a pure pair of commuting isometries $(V_1,V_2)$ is always unitarily equivalent to  isometries $(M_{\Phi},M_{\Psi}) \in \clb(H^2(\cle))$ associated to some BCL triple $(\cle,U,P)$. In particular, $\Phi(z) = (P+zP^{\bot})U^*$ and $\Psi(z) = U(P^{\bot}+zP)$ ($z \in \D$). From statement $(i)$ of Corollary \ref{pure_cor}, it is clear that $M_{\Phi}, M_{\Psi}$ are pure contractions if and only if $\Phi(0),\Psi(0)$ are pure contractions, respectively. However, the following equivalent characterization shows that the condition for pureness is actually determined by the diagonal entries of the unitary $U$.

\begin{thm}\label{compure3}
Let $(M_{\Phi},M_{\Psi}) \in \clb(H^2(\cle))$ be a pair of isometries  associated to a BCL triple $(\cle, U, P)$. Then $M_{\Phi}$, $ M_{\Psi}$ are  pure isometries if and only if $P^{\bot} U P^{\bot}|_{\mbox{ran }P^{\bot}}$, $P U^* P|_{\mbox{ran }P}$ are pure contractions, respectively.
\end{thm}
\begin{proof}
Let us begin by observing that
\begin{align*}
\langle I_{H^2(\cle)} - M_{\Phi}M_{\Phi}^* K(\cdot,\lambda)\eta, K(\cdot,\mu)\zeta \rangle &= \langle I_{H^2(\cle)} - \Phi(\mu) \Phi(\lambda)^* K(\cdot,\lambda)\eta, K(\cdot,\mu)\zeta \rangle\\
&= \langle I_{H^2(\cle)} - (P+ \mu P^{\bot})U^*U(P+\bar{\lambda}P^{\bot}) K(\cdot,\lambda)\eta, K(\cdot,\mu)\zeta \rangle\\
&= \langle I_{H^2(\cle)} - (P+ \mu \bar{\lambda}P^{\bot})K(\cdot,\lambda)\eta, K(\cdot,\mu)\zeta \rangle\\
&= (1 - \mu \bar{\lambda}) s(\mu,\lambda) \langle P^{\bot} \eta, \zeta \rangle\\
&= \langle P_{\mathbb{C}} P^{\bot} K(\cdot,\lambda)\eta, K(\cdot,\mu)\zeta \rangle.
\end{align*}
Since $\{K(\cdot,\lambda)\eta: \lambda \in \D, \eta \in \cle\}$ is a total set for $H^2(\cle)$ therefore, we obtain that \\ $ P_{H^2(\cle) \ominus M_{\Phi} H^2(\cle)} = I_{H^2(\cle)} - M_{\Phi}M_{\Phi}^* = P_{\mathbb{C}} P^{\bot}$. By definition of a fringe operator we have
\begin{align*}
F_{M_{\Psi}} = P_{H^2(\cle) \ominus M_{\Phi} H^2(\cle)} M_{\Psi} P_{H^2(\cle) \ominus M_{\Phi} H^2(\cle)} =(P_{\mathbb{C}}P^{\bot})  M_{\Psi} (P_{\mathbb{C}}P^{\bot}) = P_{\mathbb{C}} P^{\bot}UP^{\bot}.
\end{align*}
Hence, $F_{M_{\Psi}}$ is a pure contraction if and only if $P^{\bot}UP^{\bot}|_{\mbox{ran }P^{\bot}}$ is a pure contraction. In a similar manner, we have
\begin{align*}
\langle I_{H^2(\cle)} - M_{\Psi}M_{\Psi}^* K(\cdot,\lambda)\eta, K(\cdot,\mu)\zeta \rangle &= \langle I_{H^2(\cle)} - \Psi(\mu) \Psi(\lambda)^* K(\cdot,\lambda)\eta, K(\cdot,\mu)\zeta \rangle\\
&= \langle I_{H^2(\cle)} - U(P^{\bot}+ \mu P)(P^{\bot}+\bar{\lambda}P)U^* K(\cdot,\lambda)\eta, K(\cdot,\mu)\zeta \rangle\\
&= \langle I_{H^2(\cle)} - U(P^{\bot}+ \mu \bar{\lambda}P)U^*K(\cdot,\lambda)\eta, K(\cdot,\mu)\zeta \rangle\\
&= (1 - \mu \bar{\lambda}) s(\mu,\lambda) \langle UPU^* \eta, \zeta \rangle\\
&= \langle P_{\mathbb{C}}  UPU^* K(\cdot,\lambda)\eta, K(\cdot,\mu)\zeta \rangle.
\end{align*}
Therefore, $ P_{H^2(\cle) \ominus M_{\Psi} H^2(\cle)} = P_{\mathbb{C}} UPU^*$ and moreover,
\begin{align*}
F_{M_{\Phi}} = P_{H^2(\cle) \ominus M_{\Psi} H^2(\cle)} M_{\Phi} P_{H^2(\cle) \ominus M_{\Psi} H^2(\cle)} = (P_{\mathbb{C}} UPU^*) M_{\Phi} (P_{\mathbb{C}} UPU^*) = P_{\mathbb{C}} UPU^*PU^*.
\end{align*}
Since $U$ is a unitary we deduce that $F_{M_{\Phi}}$ is a pure contraction  if and only if $PU^*P|_{\mbox{ran }P}$ is a pure contraction. The proof now follows by applying Proposition \ref{compure1} to the pure pair of commuting isometries $(M_{\Phi}, M_{\Psi})$ on $H^2(\cle)$.
\end{proof}

The above theorem can be obtained from a result in the paper by Bercovici, Douglas and Foias (Theorem 3.5, \cite{BDF3}). However, our approach to this characterization is based on the general Proposition \ref{compure1}, that does not involve the BCL structure.

\section{Commuting pure contractions}\label{sec-purecont}
This section is divided into two parts. In the first part we consider certain pair of commuting contractions $(T_1,T_2)$ which have a canonical BCL triple associated to them. If we further assume $(T_1,T_2)$ to be commuting pure contractions then we find that the BCL triple satisfies some important properties. In the second part, we are interested in finding pairs $(T_1,T_2)$ for which, the associated BCL triple gives a pair of pure isometries $(M_{\Phi}, M_{\Psi})$.

Let us begin by observing that for any pair of commuting contractions $(T_1,T_2)$ we have
\begin{equation}\label{paircon1}
T_2 (I - T_1
T_1^*) T_2^* + (I - T_2 T_2^*)= (I - T_1 T_1^*) + T_1 (I - T_2
T_2^*) T_1^* .
\end{equation}
It follows that
\begin{equation}\label{norm1}
\|D_{T_1^*}T_2^*h\|^2+\|D_{T_2^*}h\|^2 =  \|D_{T_1^*}h\|^2+
\|D_{T_2^*}T_1^*h\|^2 \quad
(h\in\Hil).
\end{equation}
Thus, \[U : \{D_{T_1^*} T_2^*h \oplus D_{T_2^*} h : h \in \clh\} \raro
\{D_{T_1^*} h \oplus D_{T_2^*} T_1^* h : h \in \clh\}\]defined by
\begin{equation}\label{U-h1}
U\left(D_{T_1^*} T_2^*h, D_{T_2^*}h\right) = \left(D_{T_1^*}h, D_{T_2^*}T_1^*
h\right) \quad \quad (h \in \clh),
\end{equation}
is an isometry. Let us now introduce some subspaces
\[\clm_{U} := \{D_{T_1^*} T_2^*h \oplus D_{T_2^*} h : h \in \clh\} \subseteq \cld_{T_1^*} \oplus \cld_{T_2^*},\]
\[\cln_{U} := \{D_{T_1^*} h \oplus D_{T_2^*} T_1^* h : h \in \clh\} \subseteq \cld_{T_1^*} \oplus \cld_{T_2^*},\]
and let $P$ be the orthogonal projection onto the second component of $\cld_{T_1^*} \oplus \cld_{T_2^*}$ that is,
\[
P(h_1,h_2) = (0,h_2) \quad ((h_1,h_2) \in \cld_{T_1^*} \oplus \cld_{T_2^*}).
\]

If we now assume that $\mbox{dim}\big((\cld_{T_1^*} \oplus \cld_{T_2^*}) \ominus\clm_U \big) = \mbox{dim}\big((\cld_{T_1^*} \oplus \cld_{T_2^*}) \ominus \cln_U)$, then $U$ extends to an unitary (again denoted by $U$) on the full space $\cld_{T_1^*} \oplus \cld_{T_2^*}$. This discussion allows us to identify a canonical BCL triple associated to $(T_1,T_2)$ in this particular case.
\begin{defn}
Let $(T_1,T_2)$ be a pair of commuting contractions on $\clh$ and $P$ be an orthogonal projection onto any one of the components of $\cld_{T_1^*}\oplus \cld_{T_2^*}$. Then a BCL triple $(\cld_{T_1^*} \oplus \cld_{T_2^*},U,P)$ is said to be \textit{complete} for $(T_1,T_2)$ if the condition: $\mbox{dim}\big((\cld_{T_1^*} \oplus \cld_{T_2^*}) \ominus\clm_U \big) = \mbox{dim}\big((\cld_{T_1^*} \oplus \cld_{T_2^*}) \ominus \cln_U)$ is satisfied.
\end{defn}
\begin{rem}\label{rem1}
If $\mbox{dim }\cld_{T_i^*}<\infty$ for $i=1,2$, then $\mbox{dim}\big((\cld_{T_1^*} \oplus \cld_{T_2^*}) \ominus\clm_U \big) = \mbox{dim}\big((\cld_{T_1^*} \oplus \cld_{T_2^*}) \ominus \cln_U)$ and therefore, $(\cld_{T_1^*} \oplus \cld_{T_2^*},U,P)$ is a complete BCL triple for $(T_1,T_2)$. There are pairs of commuting contractions for which, $(\cld_{T_1^*} \oplus \cld_{T_2^*},U,P)$ is a complete BCL triple even if the defect spaces are of infinite dimension. An example will be considered in Theorem \ref{purecont2}.
\end{rem}
We have established in our discussion that the assumption: $(\cld_{T_1^*} \oplus \cld_{T_2^*},U,P)$ is a complete BCL triple for $(T_1,T_2)$ is equivalent to showing that there exists a unitary,
\begin{equation}
U=
\begin{bmatrix} A & B\\ C & D \end{bmatrix}
: \cld_{T_1^*} \oplus \cld_{T_2^*}
\raro \cld_{T_1^*} \oplus \cld_{T_2^*} .
\end{equation}
such that the following identities based on condition (\ref{U-h1}) are satisfied for all $h \in \clh$.
\begin{equation}\label{unicondn1}
D_{T_1^*}h = A D_{T_1^*}T_2^*h + B D_{T_2^*}h;
\end{equation}
\begin{equation}\label{unicondn2}
D_{T_2^*}T_1^*h = C D_{T_1^*}T_2^*h + D D_{T_2^*}h.
\end{equation}
Note that with respect to the orthogonal projection $P$ we have
\begin{equation}\label{projns}
U=
\begin{bmatrix} A & B\\ C & D \end{bmatrix}
= \begin{bmatrix} P^{\bot} U P^{\bot}|_{\mbox{ran }P^{\bot}} & P^{\bot} U P|_{\mbox{ran }P}\\ PUP^{\bot}|_{\mbox{ran }P^{\bot}} & P U P|_{\mbox{ran }P} \end{bmatrix}.
\end{equation}
Using the property $U^*U=I=UU^*$ we have
\begin{align*}
\begin{bmatrix}
A^*A+C^*C & A^*B+C^*D \\ B^*A+D^*C & B^*B+D^*D 
\end{bmatrix} = \begin{bmatrix}
I & 0 \\ 0 & I 
\end{bmatrix} = \begin{bmatrix}
AA^*+BB^* & AC^*+BD^* \\ CA^*+DB^* & CC^*+DD^*
\end{bmatrix}.
\end{align*}
If we act by $A^*$ on the left of (\ref{unicondn1}), then
\begin{align*}
A^*D_{T_1^*}h &= A^*A D_{T_1^*}T_2^*h + A^*B D_{T_2^*}h \\
&= (I - C^*C)D_{T_1^*}T_2^*h - C^*D D_{T_2^*}h\\
&= D_{T_1^*}T_2^*h - C^*(C D_{T_1^*}T_2^*h + D D_{T_2^*}h)\\
&= D_{T_1^*}T_2^*h - C^* D_{T_2^*}T_1^*h.
\end{align*}
Similarly, if we act by $D^*$ on the left of (\ref{unicondn2}), then
\begin{align*}
D^* D_{T_2^*}T_1^*h &= D^* CD_{T_1^*}T_2^*h + D^*D D_{T_2^*}h \\
&= -B^*A D_{T_1^*}T_2^*h + (I - B^*B) D_{T_2^*}h \\
&=  D_{T_2^*}h  -B^*(A D_{T_1^*}T_2^*h + B D_{T_2^*}h) \\
&=  D_{T_2^*}h -B^*D_{T_1^*}h. 
\end{align*}
Therefore, we get the following set of identities
\begin{equation}\label{unicondn3}
D_{T_2^*}h = D^* D_{T_2^*}T_1^*h + B^*D_{T_1^*}h;  
\end{equation}
\begin{equation}\label{unicondn4}
D_{T_1^*}T_2^*h = C^* D_{T_2^*}T_1^*h + A^*D_{T_1^*}h .
\end{equation}
We are now ready to identify a recurrence relation between the entries of $U$ and $(T_1,T_2)$.

\begin{lemma}\label{condn1}
For all $h \in \clh$ and $m \in \Nat \setminus \{0\}$ we have
\begin{enumerate}
\item[(i)] $A^{*m}D_{T_1^*}h = D_{T_1^*}T_2^{*m}h - \sum_{k=0}^{m-1} A^{*k} C^* D_{T_2^*}T_1^* T_2^{*(m-1-k)}h$.
\item[(ii)] $D^{m}D_{T_2^*}h = D_{T_2^*}T_1^{*m}h - \sum_{k=0}^{m-1} D^{k} C D_{T_1^*}T_2^* T_1^{*(m-1-k)}h$.
\end{enumerate}
\end{lemma}
\begin{proof}
We shall prove by the method of induction. It is clear from (\ref{unicondn4}) that $(i)$ is satisfied for $m=1$ and moreover, if we replace $h$ with $T_2^{*(m-1)}h$ in (\ref{unicondn4}) then we get
\[
A^* D_{T_1^*}T_2^{*{m-1}}h = D_{T_1^*}T_2^{*m}h - C^* D_{T_2^*}T_1^* T_2^{*(m-1)}h  \quad (m \in \Nat \setminus \{0\})
\]
 Now, let us assume that $(i)$ is satisfied for $m=n-1$ for some $n \in \Nat \setminus \{0\}$ that is,
\[
A^{*(n-1)}D_{T_1^*}h = D_{T_1^*}T_2^{*(n-1)}h - \sum_{k=0}^{n-2} A^{*k} C^* D_{T_2^*}T_1^* T_2^{*(n-2-k)}h.
\]
Then, for $m = n$ we have
\begin{align*}
A^{*n}D_{T_1^*}h &= A^* (A^{*(n-1)}D_{T_1^*}h) \\
&= A^* D_{T_1^*}T_2^{*{n-1}}h - \sum_{k=0}^{n-2} A^{*k+1} C^* D_{T_2^*}T_1^* T_2^{*(n-2-k)}h\\
&= D_{T_1^*}T_2^{*n}h - C^* D_{T_2^*}T_1^* T_2^{*(n-1)}h - \sum_{k=0}^{n-2} A^{*k+1} C^* D_{T_2^*}T_1^* T_2^{*(n-2-k)}h\\
&=  D_{T_1^*}T_2^{*n}h - \sum_{k=0}^{n-1} A^{*k} C^* D_{T_2^*}T_1^* T_2^{*(n-1-k)}h.
\end{align*}
This proves that the statement $(i)$ is true for all $m \in \Nat$. The proof for statement $(ii)$ follows in a similar manner by considering condition (\ref{unicondn2}).
\end{proof}

We are now ready to obtain our first result on the interplay between pairs of commuting pure contractions and their corresponding BCL triples.
\begin{thm}\label{purecrit}
Let $(T_1,T_2)$ be a pair of commuting pure contractions on $\clh$ such that $(\cld_{T_1^*} \oplus \cld_{T_2^*},U,P)$ is a complete BCL triple for $(T_1,T_2)$. Then $(P^{\bot}UP^{\bot}|_{\mbox{ran } P^{\bot}},PU^*P|_{\mbox{ran } P})$ is a pair of completely non co-isometric contractions.
\end{thm}
\begin{proof}
In view of the discussion in Section 2, about completely non co-isometric contractions, we will have to show that $\clh_{i}(P^{\bot}U^*P^{\bot}|_{\mbox{ran } P^{\bot}})=0$ and $\clh_{i}(P U P|_{\mbox{ran }P})=0$. Let us begin by observing that $U$ is a unitary operator implies that $A^*=P^{\bot}U^*P^{\bot}|_{\mbox{ran }P^{\bot}}$ is a contraction. Moreover, $\|A^{*m}h\|^2 = \|A^{*(m-1)}h\|^2 - \|B^*A^{*m-1}h\|^2$ for all $h \in \cld_{T_1^*}$ and $m \in \Nat \setminus \{0\}$. In particular, $h \in \clh_{i}(A^*)$ implies that  $\|A^{*m}h\| = \|h\|$ for all $m \in \Nat$ and thus $\|B^*A^{*m}h\|=0$. More precisely, $\clh_{i}(A^*)$ is a $A^*$-invariant subspace of $\mbox{ker} B^*$. Now, for any $h \in \clh_{i}(A^*)$, $D_{T_1^*}\eta \in \cld_{T_1^*}$ and $m \in \Nat \setminus \{0\}$,
\begin{align*}
\langle h, D_{T_1^*} \eta \rangle = \langle A^{m}A^{*m}h, D_{T_1^*} \eta \rangle &= \langle A^{*m} h, A^{*m}D_{T_1^*} \eta \rangle\\ 
&= \langle A^{*m} h, D_{T_1^*}T_2^{*m}\eta - \sum_{k=0}^{m-1} A^{*k} C^* D_{T_2^*}T_1^* T_2^{*(m-1-k)}\eta \rangle \\
&= \langle A^{*m} h, D_{T_1^*}T_2^{*m}\eta \rangle - \sum_{k=0}^{m-1} \langle  C A^{k} A^{*m} h,  D_{T_2^*}T_1^* T_2^{*(m-1-k)}\eta \rangle\\
&= \langle A^{*m} h, D_{T_1^*}T_2^{*m}\eta \rangle - \sum_{k=0}^{m-1} \langle  C A^* A^{*(m-k-1)} h,  D_{T_2^*}T_1^* T_2^{*(m-1-k)}\eta \rangle\\
&= \langle A^{*m} h, D_{T_1^*}T_2^{*m}\eta \rangle + \sum_{k=0}^{m-1} \langle  D B^* A^{*(m-k-1)} h,  D_{T_2^*}T_1^* T_2^{*(m-1-k)}\eta \rangle\\
&= \langle A^{*m} h, D_{T_1^*}T_2^{*m}\eta \rangle. 
\end{align*}
The last but one equality follows from the identity $CA^* =  - DB^*$. Now, it follows that
\[
|\langle h, D_{T_1^*} \eta \rangle| \leq \lim_{m \raro \infty}\|h\| \|T_2^{*m}\eta\| = 0,
\]
and therefore, we have $\clh_{i}(A^*) \perp \cld_{T_1^*}$. This is a contradiction to the fact that $\clh_{i}(A^*) \subseteq  \cld_{T_1^*}$, unless $ \clh_{i}(A^*) =0 $. In a similar manner, $\clh_{i}(P U P|_{\mbox{ran }P})=0$, follows from considering the unitary
\[
\begin{bmatrix}
D^* & B^* \\ C^* & A^*
\end{bmatrix}: \cld_{T_2^*} \oplus \cld_{T_1^*}
\raro \cld_{T_2^*} \oplus \cld_{T_1^*} ,
\]
along with statement $(ii)$ of Lemma \ref{condn1}. This completes the proof.
\end{proof}

The above result has an immediate effect on the symbols of the isometries associated to the BCL triple $(\cld_{T_1^*} \oplus \cld_{T_2^*},U,P)$ in the following manner.

\begin{thm}\label{symbols_cnc}
With the assumption of Theorem \ref{purecrit}, let $(M_{\Phi},M_{\Psi})$ be the pair of BCL type isometries associated to a complete BCL triple $(\cld_{T_1^*} \oplus \cld_{T_2^*},U,P)$ for $(T_1,T_2)$. Then $(\Phi(0),\Psi(0))$ is a pair of c.n.c. partial isometries on $\cld_{T_1^*} \oplus \cld_{T_2^*}$.
\end{thm}
\begin{proof}
For the sake of computations, let us denote $\cle:= \cld_{T_1^*} \oplus \cld_{T_2^*}$. It is evident from the definition that $(\Phi(0),\Psi(0)) = (PU^*,UP^{\bot})$ is a pair of partial isometries on $\cle$. Now, for any $h \in \clh_i(\Phi(0)^*)$ and $\eta \in P^{\bot} \cle$ we have
\begin{align*}
\langle h, \eta \rangle = \langle \Phi(0) \Phi(0)^*h, \eta \rangle = \langle Ph, \eta \rangle = 0.
\end{align*}
Thus, $\clh_i(\Phi(0)^*) \subseteq P\cle$ and therefore,
\begin{align*}
\clh_i(\Phi(0)^*) &= \clh_i((PU^*)^*)\\ 
&= \{h \in P \cle: (PU^*)^n (PU^*)^{*n} h = h, \text { for all } n \in \Nat \}\\
&= \{h \in P \cle: (PU^*P)^{n-1} (PU^*P)^{*n-1} h = h, n\geq 1 \} \\
&= \clh_i(PUP|_{\mbox{ran }P}).
\end{align*}
Similarly, for any $h \in \clh_i(\Psi(0)^*)$ and $\eta \in U P \cle$ we have
\begin{align*}
\langle h, \eta \rangle = \langle \Psi(0) \Psi(0)^*h, \eta \rangle = \langle UP^{\bot}U^*h, \eta \rangle = 0.
\end{align*}
Thus, $\clh_i(\Psi(0)^*) \subseteq UP^{\bot} \cle$ and therefore,
\begin{align*}
\clh_i(\Psi(0)^*) &= \clh_i((UP^{\bot})^*)\\ 
&= \{h \in UP^{\bot} \cle: (UP^{\bot})^n (UP^{\bot})^{*n} h = h, \text { for all } n \in \Nat \}\\
&= \{h \in UP^{\bot} \cle: U(P^{\bot}UP^{\bot})^{n-1} (P^{\bot}UP^{\bot})^{*n-1} U^*h = h, n \geq 1 \} \\
&= \{h \in UP^{\bot} \cle: (P^{\bot}UP^{\bot})^{n-1} (P^{\bot}UP^{\bot})^{*n-1} U^*h = U^*h, n\geq 1 \} \\
&= U\{h \in P^{\bot} \cle: (P^{\bot}UP^{\bot})^{n-1} (P^{\bot}UP^{\bot})^{*n-1} h = h, n\geq 1 \} \\
&= U\clh_i(P^{\bot}U^*P^{\bot}|_{\mbox{ran }P^{\bot}}).
\end{align*}
From Theorem \ref{purecrit} we have $ \clh_i(PUP|_{\mbox{ran }P}) = \{0\} = \clh_i(P^{\bot}U^*P^{\bot}|_{\mbox{ran }P^{\bot}})$. This completes the proof.
\end{proof}

There are a few important corollaries which we will like to record in the sequel. 

\begin{cor}\label{defects}
In the setting of Theorem \ref{symbols_cnc}, both $\Phi(0)$ and $\Psi(0)$ possesses the wandering subspace property. 
\end{cor}
\begin{proof}
The result follows from Theorem \ref{symbols_cnc} and Proposition \ref{partial_wander}.
\end{proof}

\begin{cor}\label{symbols_cnu}
In the setting of Theorem \ref{symbols_cnc}, both $\Phi(z)$ and $\Psi(z)$ are c.n.u. contractions on $\cle$ for all $z \in \mathbb{D}$.
\end{cor}
\begin{proof}
The method in [Theorem 3.1, \cite{Fac}] shows that $\Phi(z)$ is a transfer function corresponding to the unitary
\[
U_{\Phi} = \begin{bmatrix}
PU^* & i_1 \\i_1^*U^* & 0 
\end{bmatrix}: (\cld_{T_1^*}\oplus \cld_{T_2^*}) \oplus \cld_{T_1^*} \raro (\cld_{T_1^*}\oplus \cld_{T_2^*}) \oplus \cld_{T_1^*},
\]
where $i_1: \cld_{T_1^*} \raro \cld_{T_1^*}\oplus \cld_{T_2^*}$ is the inclusion map defined by $i_1(D_{T_1^*} h) = (D_{T_1^*} h,0)$ for $h \in \clh$. Similarly, $\Psi(z)$ is a transfer function corresponding to the unitary
\[
U_{\Psi} = \begin{bmatrix}
UP^{\bot} & Ui_2 \\i_2^* & 0 
\end{bmatrix}: (\cld_{T_1^*}\oplus \cld_{T_2^*}) \oplus \cld_{T_2^*} \raro (\cld_{T_1^*}\oplus \cld_{T_2^*}) \oplus \cld_{T_2^*},
\]
where $i_2: \cld_{T_2^*} \raro \cld_{T_1^*}\oplus \cld_{T_2^*}$ is defined by $i_2(D_{T_2^*}h) = (0,D_{T_2^*}h)$. The proof now follows by applying Proposition \ref{transfer_cnu}.
\end{proof}

\subsection{Pure contractions and BCL triples}
In this subsection we are interested in identifying pairs of commuting pure contractions $(T_1,T_2)$ with complete BCL triples $(\cld_{T_1^*} \oplus \cld_{T_2^*},U,P)$ for which, the associated $M_{\Phi}$, $M_{\Psi}$ are pure isometries. Our first result is in the case of finite dimensional defect spaces.
\begin{thm}\label{purecrit2}
Let $(T_1,T_2)$ be a pair of commuting pure contractions on $\clh$ such that $\mbox{dim } \cld_{T_i^*}< \infty$ for $i=1,2$. Then the isometries $M_{\Phi}, M_{\Psi}$ associated with the complete BCL triple $(\cld_{T_1^*} \oplus \cld_{T_2^*},U,P)$ for $(T_1,T_2)$, are pure isometries.
\end{thm}
\begin{proof}
For any contraction $T$ we have $\clh_{\text{cnu}}(T) \subseteq \clh_i(T^*)$ and therefore, $\clh_i(T^*)=0$ or $\clh_i(T)=0$ implies that $T$ is a c.n.u contraction. Thus, by Theorem \ref{purecrit}, we obtain that $(P^{\bot}UP^{\bot}|_{\mbox{ran }P^{\bot}},$ $PU^*P|_{\mbox{ran }P})$ is also a pair of c.n.u. contractions. On finite dimensional Hilbert spaces the class of c.n.u. contractions coincide with the class of strict contractions (that is, $\|T\| < 1$). The proof now follows by applying Theorem \ref{compure3}.
\end{proof}

\begin{lemma}\label{kernel}
In the setting of Theorem \ref{purecrit}, 
\begin{enumerate}
\item[(a)] $lim_{m \raro \infty} \|A^{*m} \eta \| = 0$ for all $\eta \in \mbox{ker }T_1^*$.
\item[(b)] $lim_{m \raro \infty} \|D^{m} \eta \| = 0$ for all $\eta \in \mbox{ker }T_2^*$.
\end{enumerate}
\end{lemma}
\begin{proof}
Let us prove statement $(a)$. Note that a contraction $T$ always satisfies the property $ T^* D_{T^*} = D_{T}T^*$ (see \cite{NF}). Now, if $\eta \in \mbox{ker }T_1^*$ then by the intertwining property, $D_{T_1^*} \eta \in \mbox{ker }T_1^*$. Replacing $h$ with $D_{T_1^*} \eta$ in the statement $(i)$ of Lemma \ref{condn1} we get that 
\[
A^{*m} D_{T_1^*}^2 \eta = D_{T_1^*}T_2^{*m}D_{T_1^*}\eta \quad (m \in \Nat),
\]
but $D_{T_1^*}^2 \eta = (I_{\clh} - T_1T_1^*)\eta = \eta$ implies that 
$ A^{*m} \eta = D_{T_1^*}T_2^{*m}D_{T_1^*}\eta$ and therefore, \\
$lim_{m \raro \infty} \|A^{*m} \eta \| = 0$. The statement $(b)$ follows in a similar manner by considering statement $(ii)$ of Lemma \ref{unicondn1}. This completes the proof.
\end{proof}

\begin{thm}\label{purecrit3}
Let $(T_1,T_2)$ be a pair of commuting pure partial isometries on $\clh$ such that $(\cld_{T_1^*} \oplus \cld_{T_2^*},U,P)$ is a complete BCL triple for $(T_1,T_2)$. Then the isometries $M_{\Phi}, M_{\Psi}$ associated with $(\cld_{T_1^*} \oplus \cld_{T_2^*},U,P)$, are pure isometries.
\end{thm}
\begin{proof}
Since $T_i$ is a partial isometry therefore, $\cld_{T_i^*} = \mbox{ker }T_i^*$, for $i=1,2$. The proof now follows from Lemma \ref{kernel}.
\end{proof}
\begin{rem}
Let us end this section by remarking that every pure contraction extends to a pure partial isometry. This is based on an observation by Halmos and McLaughlin (see \cite{HM}). Specifically, they observed that a contraction $T$ on $\clh$ always have a partially isometric extension $M(T)$, where
\[
M(T) := \begin{bmatrix} T & D_{T^*} \\ 0 & 0 \end{bmatrix}: \clh \oplus \clh
\raro \clh \oplus \clh.
\]
It clearly follows that $M(T)$ is a pure partial isometry on $\clh \oplus \clh$ if and only if $T$ is a pure contraction on $\clh$.
\end{rem}

\section{Dilation to commuting pure isometries}\label{sec-dilcont}
In this section, we show that the classes of commuting pure contractions appearing in the previous section always dilates to commuting pure isometries. Let us begin by briefly outlining the method that was developed by Das, Sarkar and the author (in \cite{Fac}) to get BCL type isometric dilations for pure pair of commuting contractions $(T_1,T_2)$. We will work with the convention already established in Section \ref{sec-purecont}. At first, let us observe that for $T=T_1T_2$, the following equality holds
\[
I - T T^* = I - T_1 T_2 T_1^* T_2^* = (I - T_1 T_1^*) + T_1 (I - T_2
T_2^*) T_1^*.
\]
Thus, we get an isometry $V: \cld_{T^*} \raro \cld_{T_1^*} \oplus \cld_{T_2^*}$ defined by
\[
V D_{T^*} h = (D_{T_1^*}h,D_{T_2^*}T_1^*h)\quad (h\in\Hil).
\]

If $T=T_1T_2$ is a pure contraction, then we set the following isometry [cf. (\ref{dil-def})]
\begin{equation}\label{piv}
\Pi_V:= (I_{H^2} \otimes V) \Pi_T : \clh \raro H^2(\cld_{T_1^*} \oplus \cld_{T_2^*}).
\end{equation}
In [Theorem 3.1, \cite{Fac}], the authors proved that a pure pair of commuting contractions with $\mbox{dim } \cld_{T_i^*}<\infty$ for $i=1,2$,   always dilates to a pure pair of commuting isometries. This result remains valid, even if we replace the assumption of finite dimensional defect spaces with the condition that $(\cld_{T_1^*} \oplus \cld_{T_2^*},U,P)$ is a complete BCL triple for $(T_1,T_2)$. This is due to the existence of a unitary $U$ on $\cld_{T_1^*} \oplus \cld_{T_2^*}$ satisfying the condition (\ref{U-h1}). More precisely, we have the following reformulation of [Theorem 3.1, \cite{Fac}]. The proof follows verbatim and therefore, we omit it.

\begin{thm}\label{DSSthm}
Let $(T_1,T_2)$ be a pure pair of commuting contractions on $\clh$ such that $(\cld_{T_1^*} \oplus \cld_{T_2^*},U,P)$ is a complete BCL triple for $(T_1,T_2)$. Then with respect to the isometry $\Pi_V:\clh \raro H^2(\cld_{T_1^*} \oplus \cld_{T_2^*})$ we have
\[
\Pi_V T^* = M_z^* \Pi_V, \quad \Pi_V T_1^* = M_{\Phi}^* \Pi_V, 
\quad \Pi_V T_2^* = M_{\Psi}^* \Pi_V,
\]
where
\[
\Phi(z)= (P+zP^{\perp})U^* \quad \text{and} 
\quad \Psi(z)= U(P^{\perp}+zP).
\]
\end{thm}

We will now obtain the main result of this section in the case of finite dimensional defect spaces.

\begin{thm}\label{purecont1}
Let $(T_1,T_2)$ be a pair of commuting pure contractions such that $\mbox{dim } \cld_{T_i^*}< \infty$, for $i=1,2$. Then $(T_1,T_2)$ dilates to a pair of commuting pure isometries $(M_\Phi, M_{\Psi})$ on $H^2(\cld_{T_1^*} \oplus \cld_{T_2^*})$. 
\end{thm}
\begin{proof}
Since $\mbox{dim }\cld_{T_i^*}<\infty$ for $i=1,2$, therefore, $(\cld_{T_1^*} \oplus \cld_{T_2^*},U,P)$ is a complete BCL triple for $(T_1,T_2)$. It follows from Theorem \ref{purecrit2}, that $(T_1,T_2)$ is a pair of pure contractions implies that $(M_{\Phi},M_{\Psi})$ associated with the BCL triple $(\cld_{T_1^*} \oplus \cld_{T_2^*},U,P)$, is a pair of pure isometries. The proof now follows by applying Theorem \ref{DSSthm}.
\end{proof}

\begin{rem}
Based on Theorem \ref{DSSthm} and the computations in Theorem \ref{compure3}, we are able to conclude that $\mbox{dim }\mbox{ker } M_{\Phi}^* = \mbox{dim } P^{\bot} \cle = \mbox{dim } \cld_{T_1^*}$ and similarly, $\mbox{dim } \mbox{ker } M_{\Psi}^*= \mbox{dim } \cld_{T_2^*}$. 
\end{rem}

We are now ready to obtain the main result for partial isometries.
\begin{thm}\label{purecont2}
Let $(T_1,T_2)$ be a pair of commuting pure partial isometries. Then $(T_1,T_2)$ always dilates to a pair of commuting pure isometries.
\end{thm}
\begin{proof}
In view of the above theorem, let us assume that $\mbox{dim }\cld_{T_1^*} = \infty$ or $\mbox{dim }\cld_{T_2^*} = \infty$ and let $\clk$ be a infinite dimensional Hilbert space. Consider a pair of commuting pure partial isometries $(X,Y)$ in $\clb(\clk \oplus \clh)$ defined by,
\[
X = \begin{bmatrix} 0 & 0\\ 0 & T_1 \end{bmatrix}; \quad \quad 
Y = \begin{bmatrix} 0 & 0\\ 0 & T_2 \end{bmatrix}.
\]
From a simple calculation we can observe that
\[
D_{X^*} = \begin{bmatrix} I_{\clk} & 0\\ 0 & D_{{T_1}^*} \end{bmatrix}; \quad \quad D_{Y^*} = \begin{bmatrix} I_{\clk} & 0\\ 0 & D_{{T_2}^*} \end{bmatrix}.
\]
Hence, $\cld_{X^*} \oplus \cld_{Y^*} = \clk \oplus \cld_{{T_1}^*} \oplus \clk \oplus \cld_{{T_2}^*}$. Moreover, 
\[
D_{X^*}Y^* = \begin{bmatrix} 0 & 0\\ 0 & D_{{T_1}^*}{T_2}^* \end{bmatrix}; \quad \quad D_{Y^*}X^* = \begin{bmatrix} 0 & 0\\ 0 & D_{{T_2}^*}{T_1}^* \end{bmatrix}.
\]
Therefore,
\begin{align*}
&\{D_{X^*} Y^*(k,h) \oplus D_{Y^*} (k,h) : (k,h) \in \clk \oplus \clh \}\\
& = \{ (0, D_{{T_1}^*}{T_2}^* h, k, D_{{T_2}^*}h): (k,h) \in \clk \oplus \clh \}
\end{align*}
and similarly,
\begin{align*}
&\{D_{X^*} (k,h) \oplus D_{Y^*}X^* (k,h) : (k,h) \in \clk \oplus \clh \} \\
&= \{ (k, D_{{T_1}^*} h, 0, D_{{T_2}^*}{T_1}^*h): (k,h) \in \clk \oplus \clh \},
\end{align*}
showing that both the above subspaces have the same co-dimension (equal to infinity) inside \\$\cld_{X^*} \oplus \cld_{Y^*}$. Thus, if we start with an isometry 
\begin{align*}
\tilde{U} : &\{D_{X^*} Y^*(k,h) \oplus D_{Y^*} (k,h) : (k,h) \in \clk \oplus \clh \}\\
&\raro
\{D_{X^*} (k,h) \oplus D_{Y^*} X^* (k,h) : (k,h) \in \clk \oplus \clh\}
\end{align*}
defined by
\[
\tilde{U} \left(D_{X^*} Y^* (k,h), D_{Y^*} (k,h) \right) = \left( D_{X^*} (k,h), D_{Y^*}X^*(k,h) \right) \quad \quad ((k,h) \in \clk \oplus \clh),
\]
then it extends to a unitary (again denoted by $\tilde{U}$) on the full space $\cld_{X^*} \oplus \cld_{Y^*}$. In particular, we obtain that $(\cld_{X^*}\oplus \cld_{Y^*}, \tilde{U}, \tilde{P})$ is a complete BCL triple for $(X,Y)$ on $\clk \oplus \clh$. Now, by applying Theorem \ref{purecrit3} and Theorem \ref{DSSthm} we know that there exists an isometry $\tilde{\Pi}: \clk \oplus \clh \raro H^2(\cld_{X^*}\oplus \cld_{Y^*})$ such that 
\[
\tilde{\Pi} X^* = M_{\Phi}^* \tilde{\Pi} ; \quad \tilde{\Pi} Y^* = M_{\Psi}^* \tilde{\Pi}, 
\] 
where $\Phi(z) = (\tilde{P} + z\tilde{P}^{\bot})\tilde{U}^*$ and $\Psi(z) = \tilde{U}(\tilde{P}^{\bot}+z\tilde{P})$, and $(M_{\Phi},M_{\Psi})$ is a pair of commuting pure isometries on $H^2(\cld_{X^*}\oplus \cld_{Y^*})$. Now, let us consider the inclusion map $i : \clh \raro \clk \oplus \clh$ defined by $i(h) = (0,h)$ for $h \in \clh$. Define the isometry $\Pi: = \tilde{\Pi} i: \clh \raro H^2(\cld_{X^*}\oplus \cld_{Y^*})$, then for all $h \in \clh$ we can observe that
\[
\Pi T_1^* h = \tilde{\Pi} i(T_1^* h) = \tilde{\Pi} (0,T_1^* h) 
 = \tilde{\Pi} X^* i(h) = M_{\Phi}^* \tilde{\Pi} i(h) =  M_{\Phi}^* \Pi h.
\]
In a similar manner we obtain that
\[
\Pi T_2^* h = \tilde{\Pi} i(T_2^* h) = \tilde{\Pi} (0,T_2^* h) 
= \tilde{\Pi} Y^* i(h) = M_{\Psi}^* \tilde{\Pi} i(h) =  M_{\Psi}^* \Pi h.
\]
In particular, we show the existence of an isometry $\Pi:\clh \raro H^2(\cld_{X^*}\oplus \cld_{Y^*})$ such that
\[
\Pi T_1^* = M_{\Phi}^* \Pi ; \quad \Pi T_2^* = M_{\Psi}^* \Pi.
\]
This completes the proof.
\end{proof}

The method used in proving the above theorem shows that it is sufficient to consider the dilation problem for only pair of commuting contractions $(T_1,T_2)$ with complete BCL triples. However, in some cases there is a direct approach to this problem. Let us exhibit one such instance through the following result.

\begin{thm}\label{normpure}
Let $(T_1,T_2)$ be a pure pair of commuting contractions on $\clh$ such that \\ $\|D_{T_1^*}h\| = \|D_{T_2^*}h\|$ for all $h \in \clh$. Then $(T_1,T_2)$ dilates to a pair of commuting pure isometries. 
\end{thm}
\begin{proof}
From condition (\ref{norm1}) and the assumption we get that $\|D_{T_1^*}T_2^*h\| =  \|D_{T_2^*}T_1^*h\|$ for all $h \in \clh$. Now let us consider the following pair of commuting contractions in $\clb(\clh \oplus \clh)$,
\[
A := \begin{bmatrix}
0 & 0 \\ 0 & T_1
\end{bmatrix}; \quad B := \begin{bmatrix}
0 & 0 \\ 0 & T_2
\end{bmatrix}.
\]
From our assumption it immediately follows that
\[
\|D_{A^*}B^* (h_1,h_2)\| = \|D_{B^*}A^* (h_1,h_2)\| \quad \big((h_1,h_2) \in \clh \oplus \clh)\big .
\]
Therefore, we can define an isometry 
\[
V: \{D_{A^*}B^* (h_1,h_2): (h_1,h_2) \in \clh \oplus \clh\} \raro \{D_{B^*}A^* (h_1,h_2): (h_1,h_2) \in \clh \oplus \clh\},
\]
by
\[
V (D_{A^*}B^* (h_1,h_2)) = D_{B^*}A^* (h_1,h_2) \quad ((h_1,h_2) \in \clh \oplus \clh).
\] 

Now $\cld_{A^*} = \clh \oplus \cld_{T_1^*}$, implies that the orthogonal complement of the subspace \\ $\{D_{A^*}B^* (h_1,h_2): (h_1,h_2) \in \clh \oplus \clh\} = \{ (0,D_{T_1^*}T_2^* h_2): (h_1,h_2) \in \clh \oplus \clh\}$ inside $\cld_{A^*}$ is infinite dimensional. Similarly, the orthogonal complement of the subspace $\{D_{B^*}A^* (h_1,h_2): (h_1,h_2) \in \clh \oplus \clh\}$ inside $\cld_{B^*}$ is infinite dimensional. Thus, we can extend
$V$ to a unitary $U:\cld_{A^*} \raro \cld_{B^*}$ which satisfies
\[
U D_{A^*}B^* = D_{B^*}A^*.
\]
It easily follows that $BD_{A^*}^2B^* = AD_{B^*}^2A^*$ and therefore, from condition (\ref{paircon1}) we obtain that $D_{A^*} = D_{B^*}$, which further implies that $D_{T_1^*} = D_{T_2^*}$. Now, if we consider the following unitary 
\[
W:= \begin{bmatrix}
0 & I \\ U & 0
\end{bmatrix}: \cld_{A^*} \oplus \cld_{B^*} \raro \cld_{A^*} \oplus \cld_{B^*},
\]
then it immediately satisfies
\[
W \left(D_{A^*} B^* (h_1,h_2), D_{B^*} (h_1,h_2) \right) = \left( D_{A^*} (h_1,h_2), D_{B^*}A^*(h_1,h_2) \right) \quad ((h_1,h_2) \in \clh \oplus \clh).
\]
In particular, we prove that $(\cld_{A^*} \oplus \cld_{B^*}, W, P)$ is a complete BCL triple for the pure pair of commuting contractions $(A,B)$. From Theorem \ref{compure3} we obtain that the pair of isometries $(M_{\Phi},M_{\Psi})$ associated to the BCL triple $(\cld_A^* \oplus \cld_B^*, W, P)$ are pure isometries. Applying Theorem \ref{DSSthm} we get that $(A,B)$ dilates to $(M_{\Phi},M_{\Psi})$ on $H^2(\cld_A^* \oplus \cld_B^*)$. Now, following a method similar to the proof of Theorem \ref{purecont2} we obtain that $(T_1,T_2)$ dilates to $(M_{\Phi},M_{\Psi})$. This completes the proof.
\end{proof}

It is interesting to observe that we had only assumed $(T_1,T_2)$ to be a pure pair of commuting contractions satisfying the norm condition. In particular, we obtain the following result.

\begin{cor}
Let $(T_1,T_2)$ be a pair of commuting contractions such that $\|\cld_{T_1^*}h\| = \|\cld_{T_2^*}h\|$ for all $h \in \clh$. Then $(T_1,T_2)$ is a pair of pure contractions if and only if $T= T_1 T_2$ is a pure contraction.
\end{cor}

\section{Discussion on von Neumann Inequality}\label{sec-vnI}
An important consequence of any result on dilation theory is the corresponding von Neumann inequality. In this section, our aim is to investigate the von Neumann inequality associated to Theorem \ref{purecont1}. For this purpose, let us first begin by recalling the result for pure pairs of commuting contractions obtained by Das, Sarkar and the author in \cite{Fac}.
\begin{thm}\label{vn_in1}
Let $(T_1,T_2)$ be a pure pair of commuting contractions on $\Hil$
such that \\$\mbox{dim~}\cld_{T_i}<\infty$, $i=1,2$. Then there exists an algebraic variety $\mathcal{V}_{U,P}$ in $\overline{\D}^2$ such that
\[
\|p(T_1,T_2)\|\le \sup_{(z_1,z_2)\in  \mathcal{V}_{U,P}}|p(z_1,z_2)| \quad \quad (p
\in \mathbb{C}[z_1, z_2]).
\]
Moreover, if $m = \dim(\cld_{T_1} \oplus \cld_{T_2})$, then there
exists a pure pair of commuting isometries $(M_{\Phi}, M_{\Psi})$ on
$H^2_{\mathbb{C}^m}(\mathbb{D})$ such that
\[
\mathcal{V}_{U,P} = \{(z_1, z_2) \in \overline{\D}^2: \det(\Phi(z_1
z_2) - z_1 I) =0 \mbox{~and~} \det(\Psi(z_1 z_2)- z_2 I)=0\}.
\]
\end{thm}

Since the defect spaces are finite dimensional therefore, $(\cld_{T_1^*} \oplus \cld_{T_2^*},U,P)$ is a complete BCL triple for $(T_1,T_2)$. The pure pair of commuting isometries $(M_{\Phi}, M_{\Psi})$ on
$H^2(\mathbb{C}^m)$ appearing in the above theorem is associated to the BCL triple $(\cld_{T_1^*} \oplus \cld_{T_2^*},U,P)$. Now, if we assume $(T_1,T_2)$ to be a pair of pure contractions on $\clh$ then by Theorem \ref{purecont1} we obtain that $(M_{\Phi}, M_{\Psi})$ is a pair of pure isometries. Following this, from statement $(i)$ of Corollary \ref{pure_cor} we obtain that 
$(\Phi(0)=PU^*, \Psi(0)=UP^{\bot})$ is a pair of pure contractions on $\cle$. It now follows from Proposition \ref{transfer_cnu} that for all $z\in\D$,  $\Phi(z)$ and $\Psi(z)$ are matrix-valued c.n.u. contractions and therefore, do not have any unimodular eigenvalue. Thus,
\[
\mathcal{V}_{U,P} \cap \{ (\D\times \mathbb{T}) \cup (\mathbb{T}\times \D)
\}=\emptyset,
\]
and hence
\[
\mathcal{V}_{U,P} \cap \partial \D^2= \mathcal{V}_{U,P} \cap \mathbb{T}^2.
\]
Thus, we can replace the algebraic variety $\mathcal{V}_{U,P}$ in Theorem
~\ref{vn_in1} by an algebraic distinguished variety (see \cite{AM1})
\[
\tilde{\mathcal{V}}_{U,P} = \tilde{\mathcal{V}}_1\cap\tilde{V}_2,
\]
in $\D^2$, where
\[
\tilde{\mathcal{V}}_1= \{(z_1,z_2)\in\D^2:
\mbox{det}(\Phi(z_1z_2)-z_1I)=0\},
\]
and
\[
\tilde{\mathcal{V}}_2=\{(z_1,z_2)\in\D^2:
\mbox{det}(\Psi(z_1z_2)-z_2I)=0\}.
\]
From the above discussion we obtain the following characterization.
\begin{thm}\label{vn_in2}
Let $(T_1,T_2)$ be a pair of commuting pure contractions on $\Hil$
such that \\$\mbox{dim~}\cld_{T_i}<\infty$, $i=1,2$. Then there exists an algebraic distinguished variety $\tilde{\mathcal{V}}_{U,P}$ in $\D^2$ such that
\[
\|p(T_1,T_2)\|\le \sup_{(z_1,z_2)\in  \tilde{\mathcal{V}}_{U,P}}|p(z_1,z_2)| \quad \quad (p
\in \mathbb{C}[z_1, z_2]).
\]
\end{thm}

There already exists a version of von Neumann inequality for a pair of commuting pure contractions on $\clh$ with finite dimensional defect spaces. This result was first discovered for pairs of commuting strictly contractive matrices by Agler and McCarthy in \cite{AM1}. The general theorem was proved by Das and Sarkar \cite{DS}, which we shall now state.
\begin{thm}\label{vn_in3}
Let $(T_1,T_2)$ be a pair of commuting pure contractions on $\Hil$
such that \\$\mbox{dim~}\cld_{T_1^*}<\infty$. Then there exists an distinguished variety $\mathcal{V}_{\tau_U}$ in $\D^2$ such that
\[
\|p(T_1,T_2)\|\le \sup_{(z_1,z_2)\in  \mathcal{V}_{\tau_U}}|p(z_1,z_2)| \quad \quad (p
\in \mathbb{C}[z_1, z_2]).
\]
\end{thm}
Specifically, $\tau_U(z)$ is the transfer function associated to the unitary $U$ appearing in Theorem \ref{DSSthm} and $ \mathcal{V}_{\tau_U} = \{(z_1, z_2) \in \D^2: \det(\tau_U(z_1) - z_2 I_{\cld_{T_1^*}}) =0 \}$.

A natural question that arises in this context is as follows: what is the relationship between the distinguished varieties $\tilde{\mathcal{V}}_{U,P}$ and $\mathcal{V}_{\tau_U}$? In a recent work (Theorem 4.2, \cite{BKS}) Bhattacharyya, Sau and Kumar have proved that $\tilde{\mathcal{V}}_{U,P}$ and $\mathcal{V}_{\tau_U}$ are indeed the same distinguished variety inside $\D^2$. Thus, we recover the von Neumann inequality obtained in \cite{DS} and also in \cite{AM1}.
\vspace{2mm}

\noindent\textbf{Acknowledgement:} The author would like to thank Prof. Gautam Bharali for going through the initial drafts of this article and providing valuable comments and suggestions. The author is supported by DST-INSPIRE Faculty Fellowship No. - DST/INSPIRE/04/2019/000769.

\end{document}